\documentclass{article}
\usepackage[english]{babel}
\usepackage[utf8]{inputenc}

\usepackage{setspace}

\usepackage{amssymb}
\usepackage{graphicx,color} 
\usepackage{epsfig}
\usepackage{titlesec}
\usepackage{sectsty}
\usepackage{a4,amsfonts,latexsym,amscd,amsopn,amssymb,amsmath, amsthm,bbm,stmaryrd,wasysym,mathbbol}
\usepackage{amstext}
\usepackage{verbatim}
\usepackage{catchfile}
\usepackage{autobreak}

\newtheorem{Theorem}{Theorem}[section]

\newtheorem{lemma}[Theorem]{Lemma}

\theoremstyle{definition}

\providecommand{\keywords}[1]
{
  \small	
  \textit{Keywords:} #1
}

\title{Diffusion Approximation for Transport Equations with Dissipative Drifts}

\author{First Author Name$^{a}$$^{*}$, Second Author Name$^{b}$$^{c}$, etc.$^{a}$$^{c}$ \\
        \small $^{a}$Department, University, City, Country \\
        \small $^{b}$Department, University, City, Country \\
        \small $^{c}$Department, University, City, Country \\\\
        \small $^{*}$Corresponding author: first name, initials, surname; \tt{email.address}
}
\author{
Luca Di Persio \footnote{College of Mathematics, 
Department of Computer Science, University of Verona, strada le Grazie 15 - 37134 Verona - Italy, luca.dipersio@univr.it}
  \and  Yuri Kondratiev \footnote{Department of Mathematics, Bielefeld University, 33615 Bielefeld - Germany, kondrat@math.uni-bielefeld.de}
  \and Viktorya Vardanyan \footnote{Department of Mathematics, University of Trento, Via Sommarive 14-38123 Povo(TN)-Italy, viktorya.vardanyan@unitn.it}
}

\date{} 

\begin{document}

\maketitle

\begin{abstract} 
We study stochastic differential equations(SDEs) with a small perturbation parameter. Under the dissipative condition on the drift coefficient and the local Lipschitz condition on the drift and diffusion coefficients we prove the existence and uniqueness result for the perturbed SDE, also the convergence result for the solution of the  perturbed system to the solution of the unperturbed system when the perturbation parameter approaches zero.We consider the application of the above-mentioned results to the Cauchy problem and the transport equations.

\end{abstract} \hspace{10pt}


\section{Introduction}
We consider Markov processes  $X_t^\epsilon$
which arise from small random perturbations of
dynamical systems, imposing specific  conditions on the  coefficients of the diffusion process, i.e., the dissipativity and dissipativity for differences for the drift and the local Lipschitz condition for all coefficients. These kind of processes arise in
different areas of natural sciences. The  concept of dissipativity comes, in particular, from physics.
Dissipative systems are systems which absorb more energy from the external world than they supply and such systems are contrasted with energy conserving systems like
Hamiltonian systems.The dissipativity of dynamical systems as it is known in modern system and control community was introduced by Willems in \cite{JW}.	\\
Freidlin and Wentzell in their  book  \cite{FW} have developed the theory for random perturbations assuming that the coefficients satisfy Lipschitz condition and have a linear growth bound.
They study the random perturbations by direct probabilistic methods and then deduce consequences concerning the corresponding
problems for partial differential equations. 
They consider mainly schemes of random perturbations of the form
\begin{equation}
\label{diffusion1}
\dot{X_t^\epsilon} = b(X_t^\epsilon,\epsilon \xi_t),\ X_0^\epsilon=x
\end{equation}

where $\xi_t(\omega)$, t$\geq$0 is a random process on a probability space with values in $\mathbb{R}^l$, its trajectories are right continuous, bounded and have at most a finite number of points of discontinuity on every interval $[0,T], T< \infty $.
At the points of discontinuity of $\xi_t$, where as a rule, \eqref{diffusion1} can not be satisfied, is imposed the requirement of continuity of $X_t^\epsilon$.
Additionally $\epsilon$  is a small numerical number and $b(x,y)=(b^1(x,y),...,b^r(x,y)), x \in R^r, y \in R^l$ is a vector field assumed to be jointly continuous in its variables.
Let $b(x,0)=b(x)$,  the random process $X_t^\epsilon$ is considered as a result of small perturbations of the system
\begin{equation}
\label{zeroth}
	\dot{x_t}=b(x_t),\ x_0=x
\end{equation}

The following equation

\begin{equation}
\dot{X_t^\epsilon}= b(X_t^\epsilon) + \epsilon \sigma(X_t^\epsilon) \dot{w_t}, \ X_0^\epsilon =x
\end{equation}

can be considered as a special case of \eqref{diffusion1} with $b(x,y)=b(x)+ \sigma(x) y$. Here y is substituted by white noise process.\\
The precise meaning of (\ref{diffusion2}) can be formulated in the language of stochastic integrals in the following way:
\begin{equation}
\label{integral form}
    X_t^\epsilon=x+\int_0^t b(X_s^\epsilon)ds+\epsilon \int_0^t\sigma(X_s^\epsilon) dw_s
\end{equation}
Every solution of (\ref{integral form}) is a Markov process (a diffusion process with drift vector $b(x)$ and diffusion matrix $\epsilon^2 \sigma(x)\sigma^*(x)$).\\
Freidlin and Wentzell in their  book \cite{FW} show that $X^\epsilon_t$ converges to the solution $x_t$ of the unperturbed system as $\epsilon \to 0$, moreover they discuss the application of this result to related  partial differential equations. 
Particularly, they  obtain results concerning the behaviour of solutions of boundary value problems as $\epsilon \to 0$ from the behaviour of $X^\epsilon_t(w)$ as $\epsilon \to 0$. In the theory of differential equations of parabolic type, much attention is devoted to the study of the behaviour, as $\epsilon \to 0$, for solutions of boundary value problems for equations of the form
$$  \frac{\partial v^\epsilon}{\partial t}= L^\epsilon v^\epsilon + c(x)v^\epsilon+g(x). $$
Here $L^\epsilon$ is a differential operator with a small parameter at the derivatives of highest order:
$$ L^\epsilon= \frac{\epsilon^2}{2} \sum_{i,j=1}^{r} a^{ij}(x) \frac{\partial^2}{\partial x^i x^j}+ \sum_{i=1}^{r} b^i(x) \frac{\partial}{\partial x^i}.$$
Every  operator $L^\epsilon$ (whose coefficients are assumed to be sufficiently regular) has an associated  diffusion process $X_t^{\epsilon,x}$.
This diffusion process can be given by means of the stochastic equation
\begin{equation}
\label{(3.1)}
\dot{X}_t^{\epsilon,x}=b(X_t^{\epsilon,x})+ \epsilon \sigma(X_t^{\epsilon,x})\dot{\omega}_t, \ X_0^{\epsilon,x}=x,
\end{equation}
where $\sigma(x)\sigma^{*}(x)=(a^{ij}(x)),\ b(x)=(b^1(x),...,b^r(x)).$\\
In particular , they consider the Cauchy problem:
\begin{equation}
\label{(3.2)}
\frac{\partial v^\epsilon(t,x)}{\partial t} = L^\epsilon v^\epsilon(t,x) + c(x) v^\epsilon (t,x) +g(x), \ v^\epsilon(0,x)=f(x),
\end{equation}
$t>0, \ x \in \mathbb{R}^r$ for $\epsilon>0$ and together with it the problem for the first-order operator which is obtained for $\epsilon=0$:
\begin{equation}
\label{(3.3)}
\frac{\partial v^0(t,x)}{\partial t} = L^0 v^0(t,x) + c(x) v^0 (t,x) +g(x), \ v^0(0,x)=f(x)
\end{equation}
A special case of Cauchy equation is the so called transport equation:
$$\frac{\partial v^\epsilon(t,x)}{\partial t} = L^\epsilon v^\epsilon(t,x), \ v^\epsilon(0,x)=f(x),$$
which equals \eqref{(3.2)} in the case that g $\equiv$0 and c $\equiv$0.

\section{The model}
Let $(\Omega, \mathcal{F},(\mathcal{F}_t)_{t\in [0,T]}, \mathbb{P})$ be the reference filtered probability space  and $w$ is a given l-dimensional standard Brownian motion adapted to the defined filtration ($\mathcal{F}_t)_{t \in [0,T]}$, $0<T<+\infty$ being a finite horizon time. Here $\Omega$ is a nonempty set, which is interpreted as the space of elementary events.
The second object, $\mathcal{F}$, is a $\sigma$-algebra of subsets of  $\Omega$. Finally, $\mathbb{P}$ is a probability
measure on the $\sigma$-algebra $\mathcal{F}$. \\
We consider
\begin{equation}
\label{zeroth1}
	\dot{x_t}=b(x_t),\ x_0=x
\end{equation}
and the perturbed stochastic differential equation
\begin{equation}
    \label{diffusion2}
    \dot{X_t^\epsilon}= b(X_t^\epsilon) + \epsilon \sigma(X_t^\epsilon) \dot{w_t}, \ X_0^\epsilon =x
\end{equation}
in $\mathbb{R}^r$.Where $\epsilon$ is a small numerical number, $b(x)=(b^1(x),...,b^r(x))$ is a vector field in $\mathbb{R}^r$ and $\sigma(x)=(\sigma_j^i(x))$ is a matrix having $l$ columns and $r$ rows. By a solution of this equation we understand a random process $X_t = X_t(w)$ which satisfies the relation
\[
 X_t^\epsilon=x+\int_0^t b(X_s^\epsilon)ds+\epsilon \int_0^t\sigma(X_s^\epsilon) dw_s,
\]
with probability 1 for every $t \in [0,T]$.
We usually assume that the coefficients of our diffusion  fulfil a Lipschitz condition and have a linear growth bound. Under those conditions is proved that the solution to the stochastic differential equation exists and is unique. We modify the conditions on coefficients and prove  that the existence and uniqueness result  for the solution still holds (the proof of this result is based on the book by Gihman and Skorohod \cite{GS}) in one-dimensional case.  For multi-dimensional situation it was analysed in classical book by Stroock and Varadhan  \cite{SV}. 
We 	will assume that  $\sigma$ increases no faster than linearly and $b$ satisfies dissipativity, the coefficients of \eqref{diffusion2} satisfy a local Lipschitz condition:
for some K
\[		
<y,b(y)> + \sum_{i,j} [\sigma_j^i(y)]^2 \leq K^2(1+ |y|^2);
\]
for each N there exists an $L_N$ for which
\[
\sum_{i}|b^i(y)-b^i(z)| + \sum_{i,j} |\sigma_j^i(y)-\sigma_j^i(z)| \leq L_N|y-z|
\]	
with $|y| \leq N$, $|z| \leq N$.\\

After proving the existence and uniqueness result, we will show that the zeroth approximation for the process \eqref{diffusion2} with dissipative drift and locally Lipschitz coefficients holds, i.e  the solution  of \eqref{diffusion2}  $X_t^\epsilon$ converges to the solution of \eqref{zeroth1} $x_t$ as $\epsilon \to 0$ . The last approximation will be used to show that the solution to the Cauchy problem for $\epsilon>0$ converges to the solution for $\epsilon=0$ with weaker conditions, this convergence holds  also for Transport equation.\\
Before  stating and proving the main results, we would like to state a Gronwall-Lemma which is often used in the proofs.
\begin{lemma} \label{Gronwall-Lemma} [Gronwall] 
	Let $m(t), t \in [0,T]$, be a nonnegative function satisfying the relation \begin{equation}
	\label{Gronwall}
	m(t) \leq C + \alpha \int_{0}^{t}m(s)ds, \ t\in [0,T]
	\end{equation}
	with $C,\alpha \geq 0$. \\
	Then $m(t) \leq C e^{\alpha t},$ for $ t \in [0,T]$.
\end{lemma}

\section{Main results}
\subsection{Existence and Uniqueness of a Solution}
We aim to show that under the weaker conditions on the coefficients, i.e. dissipativity for the drift and the local Lipschitz condition for all the coefficients,  \eqref{diffusion2} has a solution and the solution is unique.
This fact may be find in \cite{SV} but we will include the proof because certain steps in this proof we will use later.
To prove the existence and uniqueness result we will need the following theorem:
\begin{Theorem}
	\label{Theorem 2 GS}
	Assume that the coefficients $b_1(x)$, $b_2(x)$, $\sigma_1(x)$, $\sigma_2(x)$ of the equations
	\begin{equation}
	\label{(4)}
	\dot{X}_{t,i}^\epsilon =b_i(X_{t,i}^\epsilon) + \epsilon \sigma_i(X_{t,i}^\epsilon)\dot{\omega}_t, \ i=1,2
	\end{equation}
satisfy Lipschitz condition and linear growth condition, i.e.,
there exists a constant K such that for $t \in [0,T]$, $x,y \in \mathbb{R}^r$, $i=1,2$
\[
|b_i(x)-b_i(y)|+|\sigma_i(x)-\sigma_i(y)|\leq K|x-y|,
\]
\[
|b_i(x)|^2+|\sigma_i(x)|^2 \leq K^2(1+|x|^2)
\]
and that for some $N>0$ with $|x^j|\leq N$ for all $j \in \mathbb{N}_0: 0 \leq j \leq r$,
$b_1(x)=b_2(x)$ and $\sigma_1(x)=\sigma_2(x)$.\\
If $X_{t,1}^\epsilon$ and $X_{t,2}^\epsilon$ are solutions of \eqref{(4)} with the same initial condition $X_{0,1}^\epsilon= X_{0,2}^\epsilon=x$, $M[x^2] < \infty$, and $\tau_i$ is the largest t for which 
$\sup_{0 \leq s \leq t, 0 \leq j \leq r } |X_{t,i}^{\epsilon,j}|\leq N$, then $P \lbrace \tau_1=\tau_2 \rbrace=1$ and $$P \lbrace \sup_{0 \leq s \leq \tau_1} |X_{t,1}^\epsilon-X_{t,2}^\epsilon|=0 \rbrace =1.$$
\end{Theorem}

\begin{proof}
Define $\gamma_1(t):= 1$, if $\sup_{0 \leq s \leq t, 0 \leq j \leq r} |X_{s,1}^{\epsilon,j}|\leq N,$ and $\gamma_1(t):=0$, \\if $\sup_{0 \leq s \leq t, 0 \leq j \leq r} |X_{s,1}^{\epsilon,j}|> N.$
Then we get
\begin{align*}
\gamma_1(t) \sum_{j} [X_{t,1}^{\epsilon,j}-X_{t,2}^{\epsilon,j}] 
&= \gamma_1(t) \int_{0}^{t} [\sum_{j} b_1^j(X_{s,1}^\epsilon)-b_2^j(X_{s,2}^\epsilon) ] ds \\
&+ \gamma_1(t) \epsilon \int_{0}^{t} [\sum_{i,j} \sigma_{i,1}^j(X_{s,1}^\epsilon)-\sigma_{i,2}^j(X_{s,2}^\epsilon) ] d\omega_s \\
&= \gamma_1(t) \int_{0}^{t} [\sum_{j} b_1^j(X_{s,1}^\epsilon)-b_2^j(X_{s,1}^\epsilon) ] ds \\
&+ \gamma_1(t) \epsilon \int_{0}^{t} [\sum_{i,j} \sigma_{i,1}^j(X_{s,1}^\epsilon)-\sigma_{i,2}^j(X_{s,1}^\epsilon) ] d\omega_s \\
&+ \gamma_1(t) \int_{0}^{t} [\sum_{j} b_2^j(X_{s,1}^\epsilon)-b_2^j(X_{s,2}^\epsilon) ] ds\\
&+ \gamma_1(t) \epsilon \int_{0}^{t} [\sum_{i,j} \sigma_{i,2}^j(X_{s,1}^\epsilon)-\sigma_{i,2}^j(X_{s,2}^\epsilon) ] d\omega_s.\\
&= \gamma_1(t) \int_{0}^{t} [\sum_{j} b_2^j(X_{s,1}^\epsilon)-b_2^j(X_{s,2}^\epsilon) ] ds\\
&+ \gamma_1(t) \epsilon \int_{0}^{t} [\sum_{i,j} \sigma_{i,2}^j(X_{s,1}^\epsilon)-\sigma_{i,2}^j(X_{s,2}^\epsilon) ] d\omega_s.
\end{align*}
Where the last step is possible, because from $\gamma_1(t)=1$ follows $b_1^j(X_{s,1}^\epsilon)= b_2^j(X_{s,1}^\epsilon)$ and $\sigma_{i,1}^j(X_{s,1}^\epsilon) =\sigma_{i,2}^j(X_{s,1}^\epsilon)$ for $s \leq t$.
Thus 
\begin{align*}
\gamma_1(t)[\sum_{j} X_{t,1}^{\epsilon,j}- X_{t,2}^{\epsilon,j} ]^2 
& \leq 2 \gamma_1(t) \Big [\int_{0}^{t} \sum_{j} [b_2^j(X_{s,1}^\epsilon)-b_2^j(X_{s,2}^\epsilon)]ds \Big ]^2  \\
&+ 2 \gamma_1(t) \epsilon^2 \Big [\int_{0}^{t} \sum_{j} [\sigma_{i,2}^j(X_{s,1}^\epsilon)-\sigma_{i,2}^j(X_{s,2}^\epsilon)]dw_s \Big ]^2.
\end{align*}
Taking into account that $\gamma_1(t)=1$ implies $\gamma_1(s)=1$ for $s \leq t$ we can write the $\gamma_1(s)$'s inside the brackets. Taking the expectation and then using the Lipschitz condition and Cauchy-Schwarz-inequality, we can show that for $\epsilon \leq 1$ there exists a constant L for which
\begin{align*}
M \Big [ \gamma_1(t)[\sum_{j} X_{t,1}^{\epsilon,j}- X_{t,2}^{\epsilon,j} ]^2 \Big ]
& \leq M \Big [4 [\int_{0}^{t} \gamma_1(s) K |X_{s,1}^\epsilon-X_{s,2}^\epsilon|ds]^2 \Big ] \\
& \leq 4 K^2 t \int_{0}^{t} M [\gamma_1(s)[\sum_{j} X_{s,1}^{\epsilon,j}- X_{s,2}^{\epsilon,j} ]^2]ds\\
& = L \int_{0}^{t} M [\gamma_1(s)[\sum_{j} X_{s,1}^{\epsilon,j}- X_{s,2}^{\epsilon,j} ]^2]ds.
\end{align*}

Now we can use Lemma \ref{Gronwall-Lemma}, for which C=0. It follows that 
$$M \Big [ \gamma_1(t)[\sum_{j} X_{t,1}^{\epsilon,j}- X_{t,2}^{\epsilon,j} ]^2 \Big ]=0.$$

Considering the continuity of $X_{t,1}^\epsilon$ and $X_{t,2}^\epsilon$ we can establish
$$ P \lbrace \sup_{0 \leq t \leq T} \gamma_1(t) [\sum_{j} X_{t,1}^{\epsilon,j}-X_{t,2}^{\epsilon,j}]^2 =0 \rbrace = 
P \lbrace \sup_{0 \leq t \leq T} \gamma_1(t)  |X_{t,1}^{\epsilon}-X_{t,2}^{\epsilon}|^2 =0 \rbrace =1 $$
On the interval $[0,\tau_1]$ the processes $X_{t,1}^\epsilon$ and $X_{t,2}^\epsilon$ coincide with probability 1. Hence $P \lbrace \tau_2 \geq \tau_1 \rbrace =1.$
Interchanging the indices 1 and 2 in the proof of the theorem, we can show analogously that $P \lbrace \tau_1 \geq \tau_2 \rbrace =1.$
\end{proof}

\begin{Theorem}
	\label{Theorem 3 GS}
	Let the coefficients of \eqref{diffusion2} be defined and measurable for $t \in [0,1]$, and satisfy the conditions 
	\begin{enumerate}
		\item  For some K
	\begin{equation}
	\label{dissipativity}
	<y,b(y)> + \sum_{i,j} [\sigma_j^i(y)]^2 \leq K^2(1+ |y|^2);
	\end{equation}
		\item 
		for each N there exists an $L_N$ for which
		\begin{equation}
		\label{local Lipschitz}
		\sum_{i}|b^i(y)-b^i(z)| + \sum_{i,j} |\sigma_j^i(y)-\sigma_j^i(z)| \leq L_N|y-z|
		\end{equation}
		with $|y| \leq N$, $|z| \leq N$.
	\end{enumerate} 
Then \eqref{diffusion2} has a unique solution in the sense that for two solutions $X_{t,1}^{\epsilon}$ and $X_{t,2}^{\epsilon}$
$$  P\lbrace \sup_{0 \leq s \leq T} |X_{t,1}^{\epsilon}- X_{t,2}^{\epsilon}|=0  \rbrace =1. $$
\end{Theorem}

\begin{proof} 
	We will first start by showing the existence and afterwards we move to the uniqueness of the solution.
	Define $x_N^i$ (the i-th component of the vector $x_N$) as $x_N^i=x^i$ for $|x^i|\leq N$ and $x_N^i=N sign(x^i)$ for $|x^i|>N$, 
	$b_N^i(y)=b^i(y)$ for $|b^i(y)| \leq N $ and
	$b_N^i(y)=N sign(b^i(y))$ for $|b^i(y)|> N $,
	$\sigma_{j,N}^i(y)=\sigma_{j}^i(y)$ for $|\sigma_{j}^i|\leq N$ and $\sigma_{j,N}^i(y)=N sign(\sigma_{j}^i(y))$ for $|\sigma_{j}^i|> N.$\\
	By $X_{t,N}^\epsilon$ we denote the solution of
	\begin{equation}
	\label{approximate-diffusion}
	\dot{X}_{t,N}^\epsilon =b_N(X_{t,N}^\epsilon)+ \epsilon \sigma_N(X_{t,N}^\epsilon)\dot{w}_t, \ X_{t,N}^\epsilon=x_N.
	\end{equation}
	
	For this equation all conditions for existence are given, because we have a growth bound depending on N and for the coefficients we also have a global Lipschitz condition.

	Let $\tau_N$ be the largest value of t for which $\sup_{0 \leq s \leq t}|X_{t,N}^\epsilon| \leq N$. Let $N^{'} >N$. Since $b_N(y)=b_{N^{'}}(y)$ and $\sigma_N(y)=\sigma_{N^{'}}(y)$ for all $|b_N^i(y)|\leq N,$ $|\sigma_{j,N}^i| \leq N,$ we can now apply Theorem \ref{Theorem 2 GS} to obtain
	$X_{t,N}^\epsilon=X_{t,N^{'}}^\epsilon$, with probability 1 for $t \in [0,\tau_N]$.\\
	Hence for $N^{'}>N$:
	$$ P \lbrace \sup_{0 \leq t \leq T}|X_{t,N}^\epsilon-X_{t,N^{'}}^\epsilon|>0 \rbrace \leq P \lbrace \tau_N > T \rbrace= P \lbrace \sup_{0 \leq t \leq T}|X_{t,N}^\epsilon|>N \rbrace. $$
	
	If we can show that the probability on the right hand side converges to zero for $N \rightarrow \infty$, then it will clearly follow that $X_{t,N}^\epsilon$ converges uniformly with probability 1 to some limit $X_t^\epsilon$ as $N \rightarrow \infty$.\\
	Going to the limit in 
	$$
	X_{s,N}^\epsilon =x_N+ \int_{0}^{t} b_N(X_{s,N}^\epsilon) ds+ \epsilon \int_{0}^{t} \sigma_N(X_{s,N}^\epsilon) dw_s$$
	we convince ourselves that $X_{t}^\epsilon$ is equal with probability 1 to the continuous solution of \eqref{diffusion2}.\\
	So to finish the proof of the existence of a solution it remains to show that 
	\begin{equation}
	\label{probability N}
	\lim_{N \rightarrow \infty}
	P \lbrace \sup_{0 \leq t \leq T}|X_{t,N}^\epsilon|>N \rbrace=0.
	\end{equation}
	To do this we first define the function $\psi(y)=\frac{1}{1+|y|^2}$ and then we
	 use the Ito formula.
	We obtain
	\begin{align*}
	&M[|X_{t,N}^\epsilon|^2 \psi(x_N)]-M[|x_N|^2\psi(x_N)] \\
	=& M\Big \lbrack  \int_{0}^{t} 2 \psi(x_N) <X_{s,N}^\epsilon,b(X_{s,N}^\epsilon)> + \epsilon^2 \psi(x_N) \sum_{k=1}^{r} \sum_{i=1}^{l} [\sigma_{i,N}^k(X_{s,N}^\epsilon)]^2 dt \Big \rbrack   \\
	 \leq & M\Big \lbrack \psi(x_N) \int_{0}^{t} 2K^2(1+|X_{s,N}^\epsilon|^2) + \epsilon^2 \psi(x_N) K^2(1+|X_{s,N}^\epsilon|^2) dt \Big \rbrack   \\
	 \leq &  \psi(x_N) (2K^2t+\epsilon^2 K^2 t)+ (2K^2+\epsilon^2K^2)\int_{0}^{t} M[\psi(x_N) |X_{s,N}^\epsilon|^2] dt  
	\end{align*}
	
	By having that we can use Lemma \ref{Gronwall-Lemma} to get
	$$	M[\psi (x_N)|X_{t,N}^\epsilon|^2] 
	\leq [\psi(x_N)(2K^2t+\epsilon^2K^2t)+|x_N|^2 \psi(x_N) ]e^{(2K^2+ \epsilon^2K^2)t}.$$
	Which means we have 
	$$M[\psi (x_N) \sup_{0 \leq t \leq T}|X_{t,N}^\epsilon|^2] \leq C_1$$
where $C_1$ is independent of N.\\
We can moreover write
\begin{align*}
P \lbrace \sup_{0 \leq t \leq T} |X_{t,N}^\epsilon|>N \rbrace 
&= P \lbrace \psi(x_N)\sup_{0 \leq t \leq T} |X_{t,N}^\epsilon|^2>N^2 \psi(x_N) \rbrace \\
&\leq P \lbrace \psi(x_N)\sup_{0 \leq t \leq T} |X_{t,N}^\epsilon|^2> \delta N^2  \rbrace + P \lbrace \psi(x_N) \leq \delta \rbrace \\
& \leq \frac{C_1}{\delta N^2}+  P \lbrace \psi(x_N) \leq \delta \rbrace,
\end{align*}
where the last inequality follows from the Chebychev inequality. \\
Consequently
$$\overline{\lim\limits_{N \rightarrow \infty}} P \lbrace \sup_{0 \leq t \leq T} |X_{t,N}^\epsilon|>N \rbrace \leq P \lbrace \psi (x_N) \leq \delta \rbrace. $$	
Since $\delta$ is an arbitrary positive number and $P \lbrace \psi(x_N)=0 \rbrace =0, \eqref{probability N}$ results from the preceding relation. This completes the proof of the existence of a solution to \eqref{diffusion2}.\\
\\	
Next we want to prove the uniqueness of the solution.
Let $X_{t,1}^\epsilon$ and $X_{t,2}^\epsilon $ be two solutions of \eqref{diffusion2}. Denoting by $\phi(t)$ the variable equal to 1 if 
$\sup_{0 \leq s \leq t}|X_{s,1}^{\epsilon,i}|\leq N$ and $\sup_{0 \leq s \leq t}|X_{s,2}^{\epsilon,i}|\leq N$ and equal to 0 otherwise. Using our second condition we can write
\begin{align*}
M|X_{t,1}^{\epsilon}-X_{t,2}^\epsilon|^2 \phi(t)
& \leq 2 M[\phi(t)(\int_{0}^{t} \sum_{i}|b^i(X_{s,1}^\epsilon)-b^i(X_{s,2}^\epsilon)| ds)^2] \\
& + 2 M[\phi(t)(\int_{0}^{t} \sum_{i,j}|\sigma_{j}^i(X_{s,1}^\epsilon)-\sigma_{j}^i(X_{s,2}^\epsilon)| dw_s)^2] \\
& \leq  2t M[(\int_{0}^{t} \phi(s) \sum_{i}|b^i(X_{s,1}^\epsilon)-b^i(X_{s,2}^\epsilon)|^2 ds)] \\
& + 2 M[(\int_{0}^{t} \phi(s) \sum_{i,j}|\sigma_{j}^i(X_{s,1}^\epsilon)-\sigma_{j}^i(X_{s,2}^\epsilon)|^2 ds)] \\
& \leq (2T+2) L_N^2 \int_{0}^{t} M(\phi(s)|X_{s,1}^\epsilon-X_{s,2}^\epsilon|^2 ds)
\end{align*}

Where we first used that $a^2+b^2 \geq 2ab$, then the Cauchy-Schwarz inequality and the properties of an Ito integral and afterwards the local Lipschitz continuity.
Then we need to use Lemma \ref{Gronwall-Lemma} with $C=0$ to get 
$M|X_{t,1}^{\epsilon}-X_{t,2}^\epsilon|^2 \phi(t)=0$, which means
$$P\lbrace X_{t,1}^\epsilon \neq X_{t,2}^\epsilon \rbrace \leq 
P\lbrace \sup_{0 \leq s \leq T} |X_{t,1}^\epsilon|>N \rbrace + P \lbrace \sup_{0 \leq s \leq T} |X_{t,2}^\epsilon|>N \rbrace.  $$
This is because  $\phi (t)$ was zero for 
$\sup_{0 \leq s \leq t}|X_{s,1}^{\epsilon,i}| >N$ or $\sup_{0 \leq s \leq t}|X_{s,2}^{\epsilon,i}|> N$.
From the continuity of $X_{t,1}^\epsilon$ and $X_{t,2}^\epsilon$ follows their boundedness. Hence the probability on the right side of this inequality tend to zero as N $\rightarrow \infty$, i.e., for all $t \in [0,T]$: $P \lbrace X_{t,1}^\epsilon= X_{t,2}^\epsilon \rbrace =1 $ from which the uniqueness follows in the sense that $P \lbrace \sup_{0 \leq t \leq T}|X_{t,1}^\epsilon-X_{t,2}^\epsilon|=0 \rbrace =1$.
\end{proof}

\subsection{Zeroth Order Approximation for Dissipative Case }

After proving the existence and uniqueness of a solution to \eqref{diffusion2} with our conditions on coefficients  we want to prove convergence of the solution  of \eqref{diffusion2}  $X_t^\epsilon$  to the solution of \eqref{zeroth1} $x_t$ as $\epsilon \to 0$ under dissipativity and dissipativity for differences for the drift vector and the local Lipschitz condition for all coefficients.

\begin{Theorem} \label{Theorem 1.2.2}
	Assume that the coefficients of \eqref{diffusion2} satisfy a local Lipschitz condition, $\sigma$ increases no faster than linearly and b satisfies dissipativity and dissipativity for the differences:
	
		\begin{enumerate}
		\item  For some K
		\begin{align}
		\label{dissipativity 2}
		<y,b(y)> + \sum_{i,j} [\sigma_j^i(y)]^2 &\leq K^2(1+ |y|^2);
		\\
		<y-z,b(y)-b(z)> &\leq K^2(1+ |y-z|^2);
		\label{dissipativity for differences}
		\end{align}
		\item 
		for each N there exists an $L_N$ for which
		\begin{equation}
		\label{local Lipschitz 2}
		\sum_{i}|b^i(y)-b^i(z)| + \sum_{i,j} |\sigma_j^i(y)-\sigma_j^i(z)| \leq L_N|y-z|
		\end{equation}
		with $|y| \leq N$, $|z| \leq N$.
	\end{enumerate} 
	
	Then for all $t>0$ and $\delta>0$ we have:
	\begin{enumerate}
		\item $M|X_t^\epsilon-x_t|\leq \epsilon^2 a(t)$, and
		\item $\lim_{\epsilon \to 0} P\lbrace \max_{0 \leq s \leq t} |X_s^\epsilon -x_s|>  \delta \rbrace=0$
	\end{enumerate}
where a(t) is a monotone increasing function, which is expressed in terms of $|x|$ and K.
	
\end{Theorem}

\begin{proof}
	We start by showing that $M|X_t^\epsilon|^2$ is bounded uniformly in $\epsilon \in [0,1]$.
	To show that we first apply Ito's formula  to get
	\begin{align*}
	(1+|X_t^\epsilon|^2)-(1+|x|^2) 
	& = \sum_{k=1}^{r} \sum_{i=1}^{l} 2\int_{0}^{t}|(X_s^\epsilon)^i| \epsilon \sigma_i^k(X_s^\epsilon)dw_s^k 
	+  \int_{0}^{t} \Big \lbrack 2<X_s^\epsilon,b(X_s^\epsilon)> \\
	&+ \epsilon^2 \sum_{k=1}^{r} \sum_{i,j=1, i=j}^{l} \sigma_i^k(X_s^\epsilon)\sigma_j^k(X_s^\epsilon) \\
	& + \epsilon^2 \sum_{k=1}^{r} \sum_{i,j=1, i \neq j}^{l} \sigma_i^k(X_s^\epsilon)\sigma_j^k(X_s^\epsilon)
	\Big \rbrack  ds.
	\end{align*}
	Applying the mathematical expectation and adding $(1+|x|^2)$ on both sides we obtain:
	
	$$1+M|X_t^\epsilon|^2= 1+|x|^2+ 2 \int_{0}^{t} M <X_s^\epsilon, b(X_s^\epsilon)>ds + \epsilon^2 \int_{0}^{t} M \sum_{i,j} [\sigma_j^i(X_s^\epsilon)]^2 ds.$$	

	Using the Cauchy-Schwarz inequality, that $\sigma$ in \eqref{diffusion2} increases no faster than linearly and the dissipativity for b, the last relation implies the estimate
	\begin{align*}
	1+M|X_t^\epsilon|^2 
	&= 1+ |x|^2+ 2\int_{0}^{t} M<X_s^\epsilon,b(X_s^\epsilon)>ds
	+ \epsilon^2 \int_{0}^{t} M \sum_{i,j} [\sigma_j^i(X_s^\epsilon)]^2 ds \\
	& \leq 1+ |x|^2+ 2 \int_{0}^{t} M [K^2(1+|X_s^\epsilon|^2)] ds
	+\epsilon^2 \int_{0}^{t} M [K^2(1+|X_s^\epsilon|^2)] ds\\
	& \leq  1+ |x|^2+ (2 K^2 + \epsilon^2 K^2) \int_{0}^{t} (1+M|X_s^\epsilon|^2)ds.
	\end{align*}
	
	Next we use Lemma \ref{Gronwall-Lemma}
	we choose $m(t)=1+M|X_t^\epsilon|^2$, $C=1+|x|^2$ and $\alpha=(2K^2+\epsilon^2K^2)$.
	By doing this we obtain
	\begin{equation}
	\label{(1.5.2)}
	1+M|X_t^\epsilon|^2 \leq (1+|x|^2) \exp[(2K^2+\epsilon^2K^2)t].
	\end{equation}
	
	By the inequality we proved that $M|X_t^\epsilon|^2$ is bounded uniformly in $\epsilon \in [0,1]$.
	In the next step we want to use our result to prove that $M|X_t^\epsilon-x_t| \leq \epsilon^2 a(t)$. To do this we work through it very similarly to before.
	
	We apply the Ito formula to the function $|X_t^\epsilon-x_t|^2$, which works the same way as it did with $1+|X_t^\epsilon|$, just that the starting term vanishes, because $X_0^\epsilon=x=x_0$.
	Next we apply the mathematical expectation on both sides of the equality to get
	$$ M|X_t^\epsilon-x_t|^2 = 2 \int_{0}^{t} M <X_s^\epsilon-x_s, b(X_s^\epsilon)-b(x_s)>ds
	+\epsilon^2\int_{0}^{t}M\sum_{i,j} [\sigma_j^i(X_s^\epsilon)]^2 ds.$$
	
In the proof of the existence we proved \eqref{probability N}. Since $X_{t,N}^{\epsilon}$ converges to $X_t^\epsilon$ as $N \to \infty$ we also know
\begin{equation}
\lim\limits_{N \to \infty} P \lbrace \sup_{0 \leq s \leq T} |X_s^\epsilon| >N  \rbrace =0
\end{equation}
and
\begin{equation}
\lim\limits_{N \to \infty} P \lbrace \sup_{0 \leq s \leq T} |x_s| >N  \rbrace =0.
\end{equation}
From this follows that there exists an N, such that 
\begin{equation}
\label{prob1}
P \lbrace \sup_{0 \leq s \leq T} |X_s^\epsilon| >N  \rbrace \leq \frac{\epsilon^2}{2}
\end{equation}
and 
\begin{equation}
\label{prob2}
P \lbrace \sup_{0 \leq s \leq T} |x_s| >N  \rbrace \leq \frac{\epsilon^2}{2}.
\end{equation}

In the following calculations we first split up our mathematical expectation in two different cases, then we use the Cauchy-Schwarz inequality and \eqref{dissipativity 2}.
Afterwards we apply the local Lipschitz condition \eqref{local Lipschitz 2} and the dissipativity for the differences for b \eqref{dissipativity for differences}.Then we estimate the probabilities we used in the inequality:
\begin{align*}
    M|X_t^\epsilon-x_t|^2 
    &= 2 \int_{0}^{t} M <X_s^\epsilon-x_s, b(X_s^\epsilon)-b(x_s)>ds
	+\epsilon^2\int_{0}^{t}M\sum_{i,j} [\sigma_j^i(X_s^\epsilon)]^2 ds \\ 
	&\leq P \lbrace \max \lbrace \sup_{0 \leq s \leq T} |X_s^\epsilon|, \sup_{0 \leq s \leq T}|x_s| \rbrace \leq N \rbrace \\
	& 2 \int_{0}^{t} M \Big [ \sqrt{|X_s^\epsilon-x_s|^2\sum_{i} [b^i(X_s^\epsilon)-b^i(x_s)]^2}ds \Big |
	 \max \lbrace \sup_{0 \leq s \leq T} |X_s^\epsilon|, \sup_{0 \leq s \leq T}|x_s| \rbrace \leq N
	  \Big ]\\
	&+ P \lbrace \max \lbrace \sup_{0 \leq s \leq T} |X_s^\epsilon|, \sup_{0 \leq s \leq T}|x_s| \rbrace > N \rbrace \\
	&2 \int_{0}^{t}
	M[<X_s^\epsilon-x_s, b(X_s^\epsilon)-b(x_s)>|\max \lbrace \sup_{0 \leq s \leq T} |X_s^\epsilon|, \sup_{0 \leq s \leq T}|x_s| \rbrace > N ]
	ds \\ 
	&+\epsilon^2 K^2 \int_{0}^{t}(1+M|X_s^\epsilon|^2) ds \\	
	&\leq P \lbrace \max \lbrace \sup_{0 \leq s \leq T} |X_s^\epsilon|, \sup_{0 \leq s \leq T}|x_s| \rbrace \leq N \rbrace
	2 \int_{0}^{t} M \sqrt{|X_s^\epsilon-x_s|^2 L_N^2 |X_s^\epsilon-x_s|^2}ds \\
	&+ P \lbrace \max \lbrace \sup_{0 \leq s \leq T} |X_s^\epsilon|, \sup_{0 \leq s \leq T}|x_s| \rbrace > N \rbrace
	2 \int_{0}^{t} K^2(1+M|X_s^\epsilon-x_s|^2)ds \\
	&+\epsilon^2 K^2 \int_{0}^{t}(1+M|X_s^\epsilon|^2) ds \\ 
	&\leq
	2 L_N \int_{0}^{t} M |X_s^\epsilon-x_s|^2 ds \\
	&+ 
	(P \lbrace \sup_{0 \leq s \leq T} |X_s^\epsilon|> N \rbrace+
	P \lbrace \sup_{0 \leq s \leq T} |x_s|> N \rbrace)
	[2 K^2t+ 2 K^2 \int_{0}^{t} M|X_s^\epsilon-x_s|^2 ds] \\
	&+\epsilon^2 K^2 \int_{0}^{t}(1+M|X_s^\epsilon|^2) ds \\ 
	&\leq
	2 L_N \int_{0}^{t} M |X_s^\epsilon-x_s|^2 ds + \epsilon^2
	2 K^2t + \epsilon^2 2 K^2 \int_{0}^{t} M|X_s^\epsilon-x_s|^2 ds \\
	&+\epsilon^2 K^2 \int_{0}^{t}(1+M|X_s^\epsilon|^2) ds \\
	&\leq
	(2 L_N+\epsilon^2 2K^2) \int_{0}^{t} M |X_s^\epsilon-x_s|^2 ds  + \epsilon^2
	2 K^2t +\epsilon^2 K^2 \int_{0}^{t}(1+M|X_s^\epsilon|^2) ds.
\end{align*}
	
	We use Lemma \ref{Gronwall-Lemma} again and this time we choose $m(t)=M|X_t^\epsilon-x_t|^2, \ \alpha=(2 L_N+ \epsilon^2 2K^2), \ C= \epsilon^2 2K^2 t+
	\epsilon^2K^2 \int_{0}^{t}(1+M|X_s^\epsilon|^2)ds$.
	By this we get  
	\begin{align*}
	M|X_t^\epsilon-x_t|^2 
	&\leq e^{(2 L_N+ \epsilon^2 2K^2)t} 
	[\epsilon^2 2K^2 t+
	\epsilon^2K^2 \int_{0}^{t}(1+M|X_s^\epsilon|^2)ds]\\
	& \leq e^{(2 L_N+ \epsilon^2 2K^2)t} \epsilon^2 2 K^2 t + e^{(2 L_N+ \epsilon^2 2K^2)t} \epsilon^2 K^2 \int_{0}^{t} (1+|x|^2) \exp[(2K+\epsilon^2K^2)s] ds \\
	& \leq \epsilon^2 2 K^2 t e^{(2 L_N+ \epsilon^2 2K^2)t}  + \epsilon^2 K^2 e^{(2 L_N+ \epsilon^2 2K^2)t}
	(1+|x|^2) \int_{0}^{t}  \exp[(2K+\epsilon^2K^2)s] ds \\
	& \leq \epsilon^2 a(t).
	\end{align*}
	
	Where we used the result \eqref{(1.5.2)} and a(t) is chosen such that it is a monotone  increasing function.

Now we want to prove the second assertion of the theorem.
We will now use the Chebyshev inequality that says that $P\lbrace \xi(\omega) \geq a \rbrace \leq \frac{Mf(\xi)}{f(a)}$.
	By setting $\xi(\omega) = \max_{0 \leq s \leq t}|X_s^\epsilon-x_s|, \ a=\delta, \ f(x)=x^2$ and  applying the first assertion of the theorem we obtain
	\begin{align}
	\label{eq Cheb}
	P\lbrace \max_{0 \leq s \leq t}|X_s^\epsilon-x_s|> \delta \rbrace 
	&\leq \frac{M[\max_{0 \leq s \leq t}|X_s^\epsilon-x_s|]^2}{\delta^2} \nonumber \\
		& \leq \frac{ \epsilon^2 a(t)}{\delta^2}
	   \nonumber  \\ 
\end{align}
	Taking limits on both sides in \eqref{eq Cheb} , we get
\[
	\lim_{\epsilon \to 0} P\lbrace \max_{0 \leq s \leq t} |X_s^\epsilon -x_s|>  \delta \rbrace \leq \lim_{\epsilon \to 0} \frac{ \epsilon^2 a(t)}{\delta^2}=0
\]
\end{proof}

\subsection{Parabolic Differential equations with a Small Parameter: Cauchy Problem, Transport Equation}

We aim to obtain results concerning the behavior of solutions of Cauchy problem as $\epsilon \to 0$ from the behavior of $X^\epsilon_t(w)$ as $\epsilon \to 0$.In the preceding section we have obtained a result concerning the the behavior of solutions $X^\epsilon_t(w)$ as $\epsilon \to 0$, which will be used in the present section.
We consider the Cauchy problem:
\begin{equation}
\label{(3.2.2)}
\frac{\partial v^\epsilon(t,x)}{\partial t} = L^\epsilon v^\epsilon(t,x) + c(x) v^\epsilon (t,x) +g(x), \ v^\epsilon(0,x)=f(x),
\end{equation}
$t>0, \ x \in \mathbb{R}^r$ for $\epsilon>0$ and together with it the problem for the first-order operator which is obtained for $\epsilon=0$:
\begin{equation}
\label{(3.3.2)}
\frac{\partial v^0(t,x)}{\partial t} = L^0 v^0(t,x) + c(x) v^0 (t,x) +g(x), \ v^0(0,x)=f(x).
\end{equation}
Where $L^\epsilon$ is a differential operator with a small parameter at the derivatives of highest order:
$$ L^\epsilon= \frac{\epsilon^2}{2} \sum_{i,j=1}^{r} a^{ij}(x) \frac{\partial^2}{\partial x^i x^j}+ \sum_{i=1}^{r} b^i(x) \frac{\partial}{\partial x^i}.$$
Every  operator $L^\epsilon$ (whose coefficients are assumed to be sufficiently regular) has an associated  diffusion process $X_t^{\epsilon,x}$.
This diffusion process can be given by means of the stochastic equation
\begin{equation}
\label{(3.1)}
\dot{X}_t^{\epsilon,x}=b(X_t^{\epsilon,x})+ \epsilon \sigma(X_t^{\epsilon,x})\dot{\omega}_t, \ X_0^{\epsilon,x}=x,
\end{equation}
where $\sigma(x)\sigma^{*}(x)=(a^{ij}(x)),\ b(x)=(b^1(x),...,b^r(x)).$\\
We assume the following conditions are satisfied:
\begin{enumerate}
	\item the function c(x) is uniformly continuous and bounded for $x \in \mathbb{R}^r$;
	\item the coefficients of $L^1$ satisfy a local Lipschitz condition, b satisfies dissipativity and dissipativity for the differences;
	
	
	\item $k^{-2} \sum\lambda_t^2 \leq \sum_{i,j=1}^r a^{ij}(x) \lambda_i \lambda_j \leq k^2 \sum \lambda_i^2$ for any real $\lambda_1,\lambda_2,...,\lambda_r$ \\ and $x \in \mathbb{R}^r$, where $k^2$ is a positive constant.
\end{enumerate}
Under these conditions, the solutions of problems \eqref{(3.2.2)} and \eqref{(3.3.2)} exist and are unique.
Having these conditions we obtain the following result:
	\begin{Theorem}
	\label{Theorem Cauchy Dissipativity}
	If conditions (1)-(3) are satisfied, then the limit $\lim\limits_{\epsilon \to 0} v^\epsilon(t,x)= v^0(t,x)$ exists for every bounded continuous initial function f(x), $x \in \mathbb{R}^r$. \\
	The function $v^0(t,x)$ is a solution of problem \eqref{(3.3.2)}.	
\end{Theorem}

\begin{proof}
	If condition (3) is satisfied, then there exists a matrix $\sigma(x)$ with entries satisfying a local Lipschitz condition for which $\sigma(x)\sigma^{*}(x)=(a^{ij}(x))$. 
	The solution of \eqref{(3.2.2)} can be represented in the following way (\cite{FW} Chap.1, Sec.5):
	\begin{equation}
	\label{(3.4.2)}
	v^\epsilon(t,x)= M[f(X_t^{\epsilon,x}) \exp\lbrace \int_{0}^{t}c(X_s^{\epsilon,x})ds \rbrace] + M[\int_{0}^{t} g(X_s^{\epsilon,x}) \exp \lbrace \int_{0}^{s} c(X_u^{\epsilon,x})du \rbrace ds].
	\end{equation}
	This stays true for the changed conditions, because of the uniqueness of the solution.
	\\
	From Theorem \ref{Theorem 1.2.2} follows the  convergence of $X_s^{\epsilon,x}$ to $X_s^{0,x}$ in probability on the interval $[0,t]$ as $\epsilon \to 0$.
	Taking into account that there is a bounded continuous functional of $X_s^{\epsilon,x}(\omega)$ under the sign of mathematical expectation in \eqref{(3.4.2)},
	by the Lebesgue dominated convergence theorem, which we can use because the functional is bounded, we obtain
	\begin{align*}
	\lim\limits_{\epsilon \downarrow 0} v^\epsilon(t,x) 
	&= \lim\limits_{\epsilon \downarrow 0} M[f(X_t^{\epsilon,x}) \exp \lbrace \int_{0}^{t} c(X_s^{\epsilon,x})ds \rbrace] \\
	&+ \lim\limits_{\epsilon \downarrow 0} [\int_{0}^{t} g(X_s^{\epsilon,x})\exp \lbrace \int_{0}^{s} c(X_u^{\epsilon,x})du \rbrace ds] \\
	&= M[ \lim\limits_{\epsilon \downarrow 0} f(X_t^{\epsilon,x}) \exp \lbrace \int_{0}^{t} c(X_s^{\epsilon,x})ds \rbrace] \\
	&+ M[ \lim\limits_{\epsilon \downarrow 0} \int_{0}^{t} g(X_s^{\epsilon,x})\exp \lbrace \int_{0}^{s} c(X_u^{\epsilon,x})du \rbrace ds] \\
	&= f(X_t^{0,x}) \exp \lbrace \int_{0}^{t} c(X_s^{0,x})ds \rbrace \\
	&+\int_{0}^{t} g(X_s^{0,x})\exp \lbrace \int_{0}^{s} c(X_u^{0,x})du \rbrace ds.
	\end{align*}
	
	The function on the right side of the equality is a  solution of \eqref{(3.3.2)},this finishes the proof.
\end{proof}

The special case is when
$c(x) \equiv g(x) \equiv 0$, which gives us the Transport equation
$$\frac{\partial v^\epsilon(t,x)}{\partial t} = L^\epsilon v^\epsilon(t,x), \ v^\epsilon(0,x)=f(x),$$
the solution of the transport equation can be written in the following form
\begin{equation}
	\label{(3.4.2)}
	v^\epsilon(t,x)= M[f(X_t^{\epsilon,x})] .
	\end{equation}
	
As in the case of Cauchy problem passing to the limit when $\epsilon \to 0$ we get $\lim\limits_{\epsilon \to 0} v^\epsilon(t,x)= v^0(t,x)$ where $v^0(t,x)$ is the solution of 
$$\frac{\partial v^0(t,x)}{\partial t} = L^0 v^0(t,x), \ v^0(0,x)=f(x).$$


\end{document}